\documentclass[11pt]{article}
\usepackage{geometry,graphicx,latexsym,amssymb,amsthm}
\usepackage{tikz}\usetikzlibrary{snakes}
\usetikzlibrary{calc}
\usepackage{hyperref}
\usepackage{subfigure}

\usepackage{subfigure}

\newtheorem{thm}{Theorem}
\newtheorem{lem}[thm]{Lemma}

\newtheorem{conj}[thm]{Conjecture}

\newtheorem{clm}{Claim}

\newtheorem{subclm}{Claim}[clm]

\newcommand{\barG}{\overline{G}}
\newcommand{\gL}{\gamma_L}

\newcommand{\oc}{{\rm oc}}

\newcommand{\match}{\alpha'}

\newcommand{\smallqed}{{\tiny ($\Box$)}}

\newcommand{\1}{\vspace{0.1cm}}

\newcommand{\cC}{{\cal C}}
\newcommand{\cD}{{\cal D}}

\def\vertex(#1){\put(#1){\circle*{2}}}
\def\vertexo(#1){\put(#1){\circle{2}}}
\def\vert(#1){\put(#1){\circle*{1.5}}}
\def\verto(#1){\put(#1){\circle{1.5}}}
\def\lab(#1)#2{\put(#1){\makebox(0,0)[c]{#2}}}
\newenvironment{unnumbered}[1]{\trivlist
\item [\hskip \labelsep {\bf #1}]\ignorespaces\it}{\endtrivlist}

\begin{document}

\title{Location-Domination and Matching in Cubic Graphs}

\author{Florent Foucaud$^{1,2}$ and Michael A. Henning$^2$\thanks{Research supported in part by the South African National Research Foundation and the University of Johannesburg.}  \\
\\
$1$: LIMOS - CNRS UMR 6158 \\
Universit\'e Blaise Pascal \\
Clermont-Ferrand, France \\[2mm]
$2$: Department of Pure and Applied Mathematics \\
University of Johannesburg \\
Auckland Park, 2006 South Africa\\[2mm]
E-mails: florent.foucaud@gmail.com, mahenning@uj.ac.za}

\date{}
\maketitle

\begin{abstract}
A dominating set of a graph $G$ is a set $D$ of vertices of $G$ such that every vertex outside $D$ is adjacent to a vertex in $D$. A locating-dominating set of $G$ is a dominating set $D$ of $G$ with the additional property that every two distinct vertices outside $D$ have distinct neighbors in $D$; that is, for distinct vertices $u$ and $v$ outside $D$, $N(u) \cap D \ne N(v) \cap D$ where $N(u)$ denotes the open neighborhood of $u$. A graph is twin-free if every two distinct vertices have distinct open and closed neighborhoods. The location-domination number of $G$, denoted $\gL(G)$, is the minimum cardinality of a locating-dominating set in $G$. Garijo, Gonz\'alez and M\'arquez [Applied Math. Computation 249 (2014), 487--501] posed the conjecture that for $n$ sufficiently large, the maximum value of the location-domination number of a twin-free, connected graph on $n$ vertices is equal to~$\lfloor \frac{n}{2} \rfloor$. We propose the related (stronger) conjecture that if $G$ is a twin-free graph of order $n$ without isolated vertices, then $\gL(G)\le \frac{n}{2}$. We prove the conjecture for cubic graphs. We rely heavily on proof techniques from matching theory to prove our result.
\end{abstract}

\indent {\small \textbf{AMS subject classification: 05C69}

\section{Introduction}

A \emph{dominating set} in a graph $G$ is a set $D$ of vertices of $G$ such that every vertex outside $D$ is adjacent to a vertex in $D$. The \emph{domination number}, $\gamma(G)$, of $G$ is the minimum cardinality of a dominating set in $G$. The literature on the subject of domination parameters in graphs up to the year 1997 has been surveyed and detailed in the two books~\cite{hhs1, hhs2}.
In this paper, we focus our attention on a variation of domination, called \emph{location-domination}, which is widely studied in the literature. A \emph{locating}-\emph{dominating set} is a dominating set $D$ that locates all the vertices in the sense that every vertex outside $D$ is uniquely determined by its neighborhood in $D$. The \emph{location-domination number} of $G$, denoted $\gL(G)$, is the minimum cardinality of a locating-dominating set in $G$. The concept of a locating-dominating set was introduced and first studied by Slater~\cite{s2,s3} and studied in~\cite{cst,fh,Heia,rs,s2,s3,s4} and elsewhere.

A classic result due to Ore~\cite{o9} states that every graph without isolated vertices has a dominating set of cardinality at most one-half its order. As observed in~\cite{Heia}, while there are many graphs (without isolated vertices) which have location-domination number much larger than one-half their order, the only such graphs that are known contain many \emph{twins}, that is, pairs of vertices with the same closed or open neighborhood. Garijo, Gonz\'alez, and M\'arquez~\cite{conjpaper} consider the function $\lambda_{|\mathcal C^*}(n)$, which is the maximum value of the location-domination number of a twin-free, connected graph on $n$ vertices. They prove that for every $n \ge 14$, $\lambda_{|\mathcal C^*}(n) \ge \lfloor \frac{n}{2} \rfloor$, and they find different conditions for a twin-free graph $G$ to satisfy $\gL(G) \le \lfloor \frac{n}{2} \rfloor$. Motivated by these results, they state the following conjecture.

\begin{conj}[\cite{conjpaper}]\label{conj_original}
There exists a positive integer $n_1$ such that, for every $n \ge n_1$,  $\lambda_{|\mathcal C^*}(n) = \lfloor \frac{n}{2} \rfloor$.
\end{conj}

We pose the related conjecture that in the absence of twins, the classic bound of one-half the order for the domination number also holds for the location-domination number. 

\begin{conj}
Every twin-free graph $G$ of order $n$ without isolated vertices satisfies $\gL(G) \le \frac{n}{2}$.
 \label{conj}
\end{conj}

We remark that Conjecture~\ref{conj} implies Conjecture~\ref{conj_original}. Indeed, if Conjecture~\ref{conj} is true, then $\lambda_{|\mathcal C^*}(n) \le \lfloor \frac{n}{2} \rfloor$ for all $n \ge 2$, which implies, by the results of Garijo et al.~\cite{conjpaper}, that $\lambda_{|\mathcal C^*}(n) = \lfloor \frac{n}{2} \rfloor$ for every $n \ge 14$. Moreover, Conjecture~\ref{conj} is a stronger conjecture than Conjecture~\ref{conj_original} in the sense that Conjecture~\ref{conj} applies to twin-free graphs of arbitrary order with no isolated vertex, while Conjecture~\ref{conj_original} is claimed to hold only for (connected) twin-free graphs of sufficiently large order.~\footnote{In~\cite{Heia}, we attributed Conjecture~\ref{conj} to the authors of~\cite{conjpaper} who posed Conjecture~\ref{conj_original}. However, as correctly pointed out by the reviewers of the current paper, the statements of Conjecture~\ref{conj_original} and Conjecture~\ref{conj} are different. Hence, although Conjecture~\ref{conj} is motivated by Conjecture~\ref{conj_original}, we pose Conjecture~\ref{conj} as an independent conjecture which is a strengthening of Conjecture~\ref{conj_original}.}

Strict inequality may hold in Conjecture~\ref{conj}. Consider, for example, the
twin-free, bipartite graph $G$ formed by taking as one partite set a set $S$ of $k \ge 2$ elements, and as the other partite set all the distinct non-empty subsets of $S$, and joining each element of $S$ to those subsets it is a member of. Then, $G$ has order~$n = k + 2^k - 1$ and $\gL(G) = |S| = k = \lfloor \log_2 n \rfloor$. This is a classic construction in the area of location-domination, see for example~\cite{s3}.

Garijo et al.~\cite{conjpaper} prove Conjecture~\ref{conj} for graphs without $4$-cycles (which include trees) and for the class of graphs with independence number at least one-half the order (which includes bipartite graphs). Further, they prove Conjecture~\ref{conj} for twin-free graphs satisfying certain conditions on the upper domination number and the chromatic number. In~\cite{Heia}, the authors provide several constructions for twin-free graphs with location-domination number one-half their order. The variety of these constructions shows that these graphs have a rich structure, which is an indication that Conjecture~\ref{conj} might be difficult to prove. Further support is given to this conjecture in~\cite{Heia} where it is proved for split graphs and co-bipartite graphs, and in~\cite{Line} where it is proved for line graphs. The following theorem summarizes the known results about Conjecture~\ref{conj}.

\begin{unnumbered}{Theorem (\cite{Heia,conjpaper,hl12})}
Conjecture~\ref{conj} is true if the twin-free graph $G$ of order~$n$ (without isolated vertices) satisfies any of the following conditions. \\
\hspace*{0.1cm} {\rm (a)} {\rm (\cite{conjpaper})} $G$ has no $4$-cycles. \\
\hspace*{0.1cm} {\rm (b)} {\rm (\cite{conjpaper})} $G$ has independence number at least $\frac{n}{2}$. \\
\hspace*{0.1cm} {\rm (c)} {\rm (\cite{conjpaper})} $G$ has clique number at least $\lceil\frac{n}{2}\rceil+1$. 
\\
\hspace*{0.1cm} {\rm (g)} {\rm (\cite{conjpaper})} $G$ has upper domination number at least~$\frac{n}{2}$ or $\barG$ has upper domination number at least~$\frac{n}{2} + 1$. \\
\hspace*{0.1cm} {\rm (h)} {\rm (\cite{conjpaper})} $G$ has chromatic number at least~$\frac{3n}{4}$ or $\barG$ has chromatic number at least~$\frac{3n}{4} + 1$. \\
\hspace*{0.1cm} {\rm (d)} {\rm (\cite{Heia})} $G$ is a split graph or a co-bipartite graph.\\
\hspace*{0.1cm} {\rm (e)} {\rm (\cite{Line})} $G$ is a line graph. \\
\hspace*{0.1cm} {\rm (f)} {\rm (\cite{hl12})} $G$ is a claw-free, cubic  graph.
\end{unnumbered}

In this paper, we continue to advance the study of Conjecture~\ref{conj} by proving it for the class of cubic graphs, as stated in our main theorem:

\begin{thm}\label{thm:main}
If $G$ is a twin-free, cubic graph of order~$n$, then $\gL(G)\le \frac{n}{2}$.
 \label{t:main}
\end{thm}

We start by giving some definitions and notations in Section~\ref{sec:def-not}, and we prove Theorem~\ref{thm:main} in Section~\ref{sec:main}. The essence of our proof of Theorem~\ref{thm:main} is to apply the Tutte-Berge Formula and use matching theory in order to obtain certain desired structures of a cubic graph that will enable us to construct locating-dominating sets of size at most one-half the order of the graph.

\section{Definitions and notation}\label{sec:def-not}

For notation and graph theory terminology, we in general follow~\cite{hhs1}. Specifically, let $G$ be a graph with vertex set $V(G)$, edge set $E(G)$ and with no isolated vertex. The \emph{open neighborhood} of a vertex $v \in V(G)$ is $N_G(v) = \{u \in V \, | \, uv \in E(G)\}$ and its \emph{closed neighborhood} is the set $N_G[v] = N_G(v) \cup \{v\}$. The degree of $v$ is $d_G(v) = |N_G(v)|$. If the graph $G$ is clear from the context, we simply write $V$, $E$, $N(v)$, $N[v]$ and $d(v)$ rather than $V(G)$, $E(G)$, $N_G(v)$, $N_G[v]$  and $d_G(v)$, respectively. Two distinct vertices $u$ and $v$ of a graph $G$ are \emph{open twins} if $N(u)=N(v)$ and \emph{closed twins} if $N[u]=N[v]$. Further, $u$ and $v$ are \emph{twins} in $G$ if they are open twins or closed twins in $G$. A graph is \emph{twin-free} if it has no twins. We use the standard notation $[k] = \{1,2,\ldots,k\}$.

Given a set $F$ of edges, we will denote by $G-F$ the subgraph obtained from $G$ by deleting all edges of $F$. For a set $S$ of vertices, $G-S$ is the graph obtained from $G$ by removing all vertices of $S$ and removing all edges incident to vertices of $S$. The subgraph induced by $S$ is denoted by $G[S]$. A \emph{cycle} on $n$ vertices is denoted by $C_n$ and a \emph{path} on $n$ vertices by $P_n$.
An \emph{odd component}~\label{oddcompt} of $G$ is a component of $G$ of odd order. The number of odd components of $G$ is denoted by $\oc(G)$.

A set $D$ is a dominating set of $G$ if $N[v] \cap D \ne \emptyset$ for every vertex $v$ in $G$, or, equivalently, $N[S] = V(G)$. Two distinct vertices $u$ and $v$ in $V(G) \setminus D$ are \emph{located}  by $D$ if they have distinct neighbors in $D$; that is, $N(u) \cap D \ne N(v) \cap D$. If a vertex $u \in V(G) \setminus D$ is located from every other vertex in $V(G)\setminus D$, we simply say that $u$ is \emph{located}  by $D$.
For $k \ge 1$ if $X$ is a set of vertices in $G$ and $x \in V(G) \setminus X$, then the vertex $x$ is said to be $k$-dominated by $X$ if $x$ has exactly $k$ neighbors inside $X$; that is, $|N(x) \cap X|= k$.

A set $S$ is a \emph{locating set} of $G$ if every two distinct vertices outside $S$ are located by $S$. In particular, if $S$ is both a dominating set and a locating set, then $S$ is a locating-dominating set. Further, if $S$ is both a total dominating set and a locating set, then $S$ is a \emph{locating-total dominating set} (where $S$ is a \emph{total dominating set} of $G$ if every vertex of $G$ is adjacent to some vertex in $S$). We remark that the only difference between a locating set and a locating-dominating set in $G$ is that a locating set might have a unique non-dominated vertex.

An \emph{independent set} in $G$ is a set of vertices no two of which are adjacent. 
Two distinct edges in a graph $G$ are \emph{independent} if they are not adjacent in $G$ (i.e., the two edges are not incident with a common vertex).  A set of pairwise independent edges of $G$ is called a \emph{matching} in $G$.
A matching of maximum cardinality in $G$ is called a \emph{maximum matching} in $G$. The number of edges in a maximum matching of a graph $G$ is called the \emph{matching number} of $G$, denoted by $\match(G)$. Let $M$ be a specified matching in a graph $G$. A vertex $v$ of $G$ is an \emph{$M$-matched vertex} if $v$ is incident with an edge of $M$; otherwise, $v$ is an \emph{$M$-unmatched vertex}. If the matching $M$ is clear from context, we simply call a $M$-matched vertex a \emph{matched vertex} and a $M$-unmatched vertex an \emph{unmatched vertex}.

\section{Proof of Theorem~\ref{thm:main}}\label{sec:main}

In this section, we present a proof of Theorem~\ref{thm:main}. Our proof relies heavily on matching theory in graphs. We begin with some useful definitions and lemmas related to matchings.

\subsection{Useful definitions and lemmas}

We shall need the following theorem of Berge~\cite{Berge} about the matching number of a graph, which is sometimes referred to as the Tutte-Berge formulation for the matching number. Recall that $\oc(G)$ denotes the number of odd components in a graph $G$.

\begin{thm}{\rm (Tutte-Berge Formula)}
For every graph $G$,
\[
\match(G) = \min_{X \subseteq V(G)} \frac{1}{2} \left( |V(G)| +
|X| - \oc(G - X)  \right).
\]
\label{Berge}
\end{thm}

We shall also need the following structural result about maximum matchings in graphs which is a consequence of the proof of the Tutte-Berge Formula.

\begin{thm}{\rm (\cite{Berge})}
Let $G=(V,E)$ be a graph and let $X$ be a proper subset of vertices
of $G$ such that $(|V|+|X|-\oc(G-X))/2$ is minimum. If $M$ is a
maximum matching in $G$, then $|M| = (|V|+|X|-\oc(G-X))/2$ and there
are exactly $\oc(G - X) - |X|$ vertices that are $M$-unmatched.
Furthermore, if $M_X$ is the subset of edges of $M$ that belong to
$G - X$, then every vertex in $G - X$ is $M_X$-matched, except for
exactly one vertex from each odd component of $G - X$. If $U$
denotes this set of $\oc(G-X)$ vertices that are $M_X$-unmatched,
one from each odd component of $G - X$, then $X$ is $M$-matched to a
subset of vertices in $U$.
 \label{t:structure}
\end{thm}

The structure described in Theorem~\ref{t:structure} is illustrated in
Figure~\ref{fig:matching}.

\begin{figure}[ht]
\centering
  \scalebox{0.8}{\begin{tikzpicture}[join=bevel,inner sep=0.6mm,line width=0.8pt, scale=0.4]

      \foreach \i in {0,1,2,4}{
        \path (0,3.5*\i) node[draw,fill, shape=circle] (x\i) {};
        \path (4,2.5*\i-3) node[draw,fill, shape=circle] (om\i) {};
        \draw[line width = 2pt] (x\i) -- (om\i);
        \draw[line width=1.2pt] (6.3,2.5*\i-3) ellipse(3.3cm and 1cm) node{};
        \foreach \j in {0,3}{
          \path (4+1+\j,2.5*\i+0.5-3) node[draw,fill, shape=circle] (om\i\j-t) {};
          \path (4+1+\j,2.5*\i-0.5-3) node[draw,fill, shape=circle] (om\i\j-b) {};
          \draw[line width = 2pt] (om\i\j-t) -- (om\i\j-b);
          \path (6.5,2.5*\i-3) node {$\ldots$};
        }
      }

      \path (x2)+(0,3) node[rotate=90] {$\ldots$};
      \path (x2)+(6.5,-2.5) node[rotate=90] {$\ldots$};
      \path (x4)+(6.5,4) node[rotate=90] {$\ldots$};
      \path (x0)+(0,-2) node {$X$};
      \draw[line width=1.2pt, draw = black, rounded corners] (-1,-3) rectangle ++(2,18.5);

      \foreach \i in {0,1,3}{
        \path (4,13+2.5*\i) node[draw,fill, shape=circle] (ou\i) {};
        \draw[line width=1.2pt] (6.3,13+2.5*\i) ellipse(3.3cm and 1cm) node{};
        \foreach \j in {0,3}{
          \path (4+1+\j,13+2.5*\i+0.5) node[draw,fill, shape=circle] (ou\i\j-t) {};
          \path (4+1+\j,13+2.5*\i-0.5) node[draw,fill, shape=circle] (ou\i\j-b) {};
          \draw[line width = 2pt] (ou\i\j-t) -- (ou\i\j-b);
          \path (6.5,13+2.5*\i) node {$\ldots$};
        }
      }

      \draw[line width=1pt, draw = black, rounded corners] (3.5,11.5) rectangle ++(1,10.5);
      \path (-5,21.6) node {$M$-unmatched vertices};
      \draw[->,line width=1.2pt] (0.5,21.5) -- (3.4,21.5);

      \draw[snake={brace},segment amplitude=5mm,mirror snake,line width=1pt] (10,-4) to (10,21.5);
      \path (16,9.5) node {odd components};
      \path (16,8.5) node {in $G-X$};


      \foreach \i in {0,1,3}{
        \draw[line width=1.2pt] (-6.5,2.5*\i) ellipse(3cm and 1cm) node{};
        \foreach \j in {0,3}{
          \path (-5-\j,2.5*\i+0.5) node[draw,fill, shape=circle] (e\i\j-t) {};
          \path (-5-\j,2.5*\i-0.5) node[draw,fill, shape=circle] (e\i\j-b) {};
          \draw[line width = 2pt] (e\i\j-t) -- (e\i\j-b);
          \path (-6.5,2.5*\i) node {$\ldots$};
          }
      }

      \path (x2)+(-6.5,-2) node[rotate=90] {$\ldots$};

      \draw[snake={brace},segment amplitude=5mm,line width=1pt] (-10,-1.5) to (-10,9);
      \path (-16,4.5) node {even components};
      \path (-16,3.5) node {in $G-X$};


      \draw[line width=0.5pt]
      (e03-b) .. controls +(0.6,-2) and +(-0.6,-2) .. (x0)
      (e00-t) .. controls +(3,0) and +(0,-2) .. (x1)
      (e10-t) .. controls +(-0.5,1.2) and +(0.5,1.2) .. (e13-t)
      (e10-b) .. controls +(3,0) and +(0,1) .. (x0)
      (e30-b) .. controls +(3,0) and +(0,1) .. (x1)
      (e30-b) .. controls +(1,1) and +(-1,1) .. (x2)
      (e30-t) .. controls +(-0.5,1.2) and +(0.5,1.2) .. (e33-t)
      (e30-b) .. controls +(-0.5,-1.2) and +(0.5,-1.2) .. (e33-b)
      (e33-t) .. controls +(1,5) and +(-3,0) .. (x4)

      (x1)--(x2)
      (om2)--(x1) (om2)--(x0) (om20-t)--(x2) (x0)--(om1) (x0)--(om00-t)
      (om00-b) .. controls +(0.5,-1.2) and +(-0.5,-1.2) .. (om03-b)
      (om0)--(om00-t) (om10-b)--(om1)--(om10-t) (om2)--(om20-b) (om4)--(om40-t)
      (om40-t) .. controls +(0.5,1.2) and +(-0.5,1.2) .. (om43-t)

      (ou3) .. controls +(-3,0) and +(0,5) .. (x4)
      (ou10-t) .. controls +(-3,2) and +(1,2) .. (x4)
      (ou00-b) .. controls +(0,-3) and +(1,2) .. (x2)
      (ou30-t) .. controls +(0.5,1.2) and +(-0.5,1.2) .. (ou33-t)
      (ou0)--(ou00-t) (ou1)--(ou10-b) (ou3)--(ou30-b);
  \end{tikzpicture}}
  \caption{Example of the structure of a graph with maximum matching $M$ (thickened edges) with respect to a given set $X$.}
  \label{fig:matching}
\end{figure}
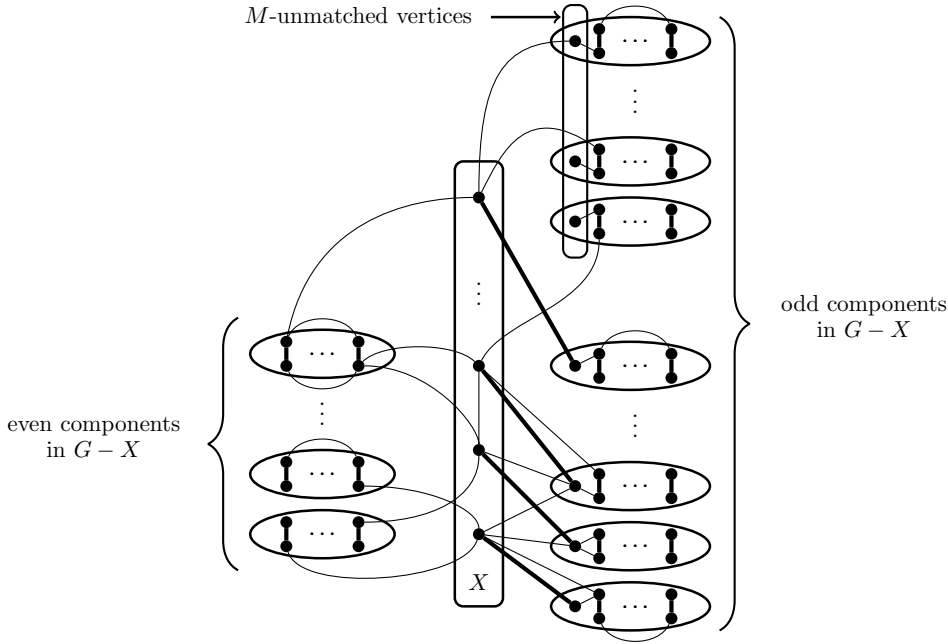

\noindent\textbf{Definition of the set $\cD_G(M)$.} Let $G$ be a graph and let $M$ be a maximum matching of $G$. We define $\cD_G(M)$ to be the collection of all sets $D$ of vertices such that the following holds: 
\begin{itemize}
\item For every edge $uv \in M$, if exactly one of $u$ and $v$ has a neighbor that is $M$-unmatched, then the vertex in $\{u,v\}$ with an $M$-unmatched neighbor belongs to $D$.
\item For every edge $uv \in M$, if neither $u$ nor $v$ has an $M$-unmatched neighbor or if both $u$ and $v$ have a (common) $M$-unmatched neighbor, then exactly one of $u$ and $v$ belongs to $D$.
\end{itemize}
If the graph $G$ is clear from the context, we simply write $\cD(M)$ rather than $\cD_G(M)$.

\noindent\textbf{Definition of a $D$-bad pair.}
Given a set $D \subseteq V(G)$, we define a \emph{$D$-bad pair} of vertices as two vertices in $V(G) \setminus D$ that are not located by $D$. If the set $D$ is clear from the context, we simply write that a pair of vertices is a \emph{bad pair} rather than a $D$-bad pair.

In our proof, we will use the following lemmas.

\begin{lem}\label{lem:M-unmatched-2dom}
Let $G$ be a cubic graph, let $M$ be a maximum matching of $G$, and let $D\in\cD_G(M)$. Then, $D$ is a dominating set of $G$, and each $M$-unmatched vertex is dominated by at least two vertices of~$D$.
\end{lem}
\begin{proof}
It follows readily from the two properties of sets $D \in \cD_G(M)$, that every $M$-matched vertex is dominated by $D$. If $x$ is an $M$-unmatched vertex, then since $G$ is cubic, the vertex $x$ is adjacent to two $M$-matched vertices that are incident with distinct edges, $e_1$ and $e_2$ say, of $M$. Hence, by the construction of $D$, the set $D$ contains a neighbor of $x$ incident with $e_1$ and a neighbor of $x$ incident with $e_2$. Thus, $x$ is dominated by at least two vertices of $D$.
\end{proof}

\begin{lem}\label{lem:2dom}
Let $G$ be a twin-free, cubic graph, let $M$ be a maximum matching of $G$, and let $D\in\cD_G(M)$. Then, the vertices of each $D$-bad pair are $2$-dominated.
\end{lem}
\begin{proof}
Let $\{u,v\}$ be a $D$-bad pair. If $u$ and $v$ were $3$-dominated by $D$, then they would be open twins, a contradiction. Hence, $u$ and $v$ are dominated by at most two vertices of $D$. By Lemma~\ref{lem:M-unmatched-2dom}, only $M$-matched vertices can be $1$-dominated by $D$, but if $u$ and $v$ are $M$-matched vertices, they would each be dominated by the vertex of $D$ that they are matched to under~$M$ and would therefore not form a $D$-bad pair, a contradiction. Hence, $u$ and $v$ are $2$-dominated by $D$.
\end{proof}

\begin{lem}\label{lem:matched}
Let $G$ be a twin-free, cubic graph. Among all maximum matchings $M$ of $G$ and all sets $D\in\cD_G(M)$, let the matching $M_0$ and the set $D_0\in\cD_G(M_0)$ be chosen so that the number of $D_0$-bad pairs is minimum. Then, the vertices of each $D_0$-bad pair are $M_0$-matched vertices.
\end{lem}
\begin{proof}Let $X$ be a proper subset of vertices of $G$ such that $(|V|+|X|-\oc(G-X))/2$ is minimum. The structure of the graph $G$ with respect to the matching $M_0$ and the set $X$ is described in Theorem~\ref{t:structure}.
Let $\{u,v\}$ be an arbitrary $D_0$-bad pair. Suppose to the contrary that they are not both $M_0$-matched vertices. By Lemma~\ref{lem:2dom}, both $u$ and $v$ are $2$-dominated by $D_0$. By the definition of a set in $\cD_G(M_0)$, if both $u$ and $v$ are $M_0$-unmatched and $2$-dominated, they would be open twins, a contradiction. Therefore, exactly one of $u$ and $v$ is $M_0$-unmatched. Renaming $u$ and $v$ if necessary, we may assume that $u$ is $M_0$-unmatched.
Thus, by Theorem~\ref{t:structure}, the vertex $u$ belongs to an odd component, $C_u$ say, of $G-X$. Let $x$ and $y$ be the two common neighbors of $u$ and $v$ in $D_0$. Let $x'$ and $y'$ be the vertices $M_0$-matched to $x$ and $y$, respectively.

Suppose that both $x$ and $y$ belong to $C_u$. If $v \notin V(C_u)$, then $v \in X$. But then the vertex that is $M_0$-matched to $v$ belongs to $D_0$, implying that $v$ would be $3$-dominated by $D_0$, contradicting Lemma~\ref{lem:2dom}. Hence, $v \in V(C_u)$. If $v$ is matched to neither $x$ nor $y$ by $M_0$, then, once again, $v$ would be $3$-dominated by $D_0$, a contradiction. Hence, $v$ is $M_0$-matched to either $x$ or $y$. Renaming $x$ and $y$ if necessary, we may assume that $xv \in M_0$, and so $v = x'$. If $u$ and $v$ are adjacent, then $u$ and $v$ would be closed twins, a contradiction. Hence, $u$ and $v$ are not adjacent. The third neighbor of $u$, different from $x$ and $y$, is therefore the vertex $y'$ that is $M_0$-matched to $y$ (otherwise, by the definition of $\cD_G(M)$, $u$ would be $3$-dominated, a contradiction). We now consider the set $D' = (D_0\setminus \{y\}) \cup \{y'\}$.

We note that $y$ and $y'$ have a common $M_0$-unmatched neighbor, namely $u$, implying that $D' \in \cD(M_0)$. Suppose there is a vertex $z$ different from $u$ that is adjacent to both $x$ and $y'$. Then, $z$ is either in $X$ or in $C_u$. In both cases, $z$ is $M_0$-matched. If $z = v$, then $u$ and $z$ are open twins, a contradiction. Since $N(y) = \{u,v,y'\}$ while $z$ is adjacent to $x$, we note that $z \ne y$. Therefore, $z \notin \{v,y\}$ and the $M_0$-matched neighbor of $z$ is in $D'$, implying that $z$ is $3$-dominated by $D'$. Hence, the vertex $u$ is the only vertex dominated only by $x$ and $y'$ in $D'$ and it is therefore located by $D'$. Moreover, both $v$ and $y$ are $1$-dominated by $D'$, and are therefore located by $D'$ by Lemma~\ref{lem:2dom}. Finally, no other vertex has been affected by the removal of $y$ from $D_0$. Hence, the number of $D'$-bad pairs is strictly less than the number of $D_0$-bad pairs, contradicting our choice of $D_0$. Therefore, at most one of $x$ and $y$ belong to $C_u$.

Suppose that exactly one of $x$ and $y$ belongs to $C_u$. Renaming $x$ and $y$ if necessary, we may assume that $x \in V(C_u)$. Then, $y \in X$. If $v \in X$ or if $v \in V(C_u) \setminus \{x'\}$, then $v$ would be $3$-dominated by $D_0$, a contradiction. Hence, $v = x'$, and so $v$ is $M_0$-matched to $x$. Since $u$ is $2$-dominated by $D_0$, the vertex $u$ is adjacent to either $v$ or $y'$. Since $y'$ belongs to a component of $G - X$ different from $C_u$, the vertex $u$ is adjacent to $v$. But then $u$ and $v$ are closed twins, a contradiction.

Therefore, both $x$ and $y$ belong to $X$. This implies that $u$, $x'$ and $y'$ belong to three different components of $G - X$. In particular, $u$ is adjacent to neither $x'$ nor $y'$, implying that the third neighbor of $u$ different from $x$ and $y$ is an $M_0$-matched vertex and therefore belongs to the set $D_0$ by the construction of sets in $\cD_G(M_0)$. Thus, $u$ is then $3$-dominated by $D_0$, a contradiction. This completes the proof of the lemma.
\end{proof}

Note that in any twin-free, cubic  graph $G$, every $4$-cycle is an induced $4$-cycle. The following structure will play an important role in our proof.

\noindent\textbf{Definition of a bad $(D,M)$-matched $4$-cycle.}
Let $C \colon u_Cu'_Cv_Cv_C'u_C$ be a $4$-cycle in a (twin-free cubic) graph $G$, $M$ a matching of $G$, and $D$ a subset of vertices of $G$. We say that $C$ is a \emph{bad $(D,M)$-matched $4$-cycle} if $u_Cu'_C \in M$, $v_Cv'_C \in M$, $D \cap V(C) = \{u'_C,v'_C\}$ and $v_C$ is adjacent to exactly two vertices of $D$ (and so, $N(v_C) \cap D = \{u'_C,v'_C\}$).

Given two bad $(D,M)$-matched $4$-cycles, $A$ and $B$, we say that $A$ is \emph{dependent} on $B$ via the vertex $u'_A$ or $v'_A$ if $u_B$ is adjacent to $u'_A$ or to $v'_A$, respectively. An illustration is given in Figure~\ref{fig:4cycles}. We note that if $A$ is dependent on $B$, then $u_B$ is $3$-dominated by $D$.

\begin{figure}[ht]
\centering
  \scalebox{1.0}{\begin{tikzpicture}[join=bevel,inner sep=0.6mm,line width=0.8pt, scale=0.4]

\path (0,0) node[draw, shape=circle] (uA) {};
\draw (uA) node[below=0.1cm] {$u_A$};
\path (uA)+(2,0) node[draw, shape=circle] (vA) {};
\draw (vA) node[below=0.1cm] {$v_A$};
\path (vA)+(1.5,0) node[draw, shape=circle] (xA) {};
\path (uA)+(0,2) node[draw, fill, shape=circle] (u'A) {};
\draw (u'A) node[above=0.1cm] {$u'_A$};
\path (vA)+(0,2) node[draw, fill, shape=circle] (v'A) {};
\draw (v'A) node[above=0.1cm] {$v'_A$};

\draw[line width=0.5pt] (uA)--(v'A) (xA)--(vA)--(u'A) (uA)--++(-1,0)  (v'A)--++(1,0);
\draw[line width=2pt] (uA)--(u'A) (vA)--(v'A);

\path (-3,0) node[draw, shape=circle] (uB) {};
\draw (uB) node[below=0.1cm] {$u_B$};
\path (uB)+(-2,0) node[draw, shape=circle] (vB) {};
\draw (vB) node[below=0.1cm] {$v_B$};
\path (uB)+(0,2) node[draw, fill, shape=circle] (u'B) {};
\draw (u'B) node[above=0.1cm] {$u'_B$};
\path (vB)+(0,2) node[draw, fill, shape=circle] (v'B) {};
\draw (v'B) node[above=0.1cm] {$v'_B$};
\path (vB)+(-1.5,0) node[draw, shape=circle] (xB) {};

\draw[line width=0.5pt] (uB)--(v'B) (xB)--(vB)--(u'B)  (u'B)--++(1,0)  (v'B)--++(-1,0) (uB) .. controls +(1,0) and +(-1,0) .. (u'A);
\draw[line width=2pt] (uB)--(u'B) (vB)--(v'B);

  \end{tikzpicture}}
  \caption{Two bad $(D,M)$-matched $4$-cycles $A$ and $B$, where $A$ is dependent
    on $B$ via $u'_A$. Edges of $M$ are thickened. Black vertices
    belong to $D$, and white vertices do not.}
  \label{fig:4cycles}
\end{figure}
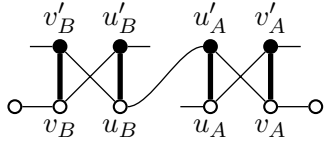

Given a set $S$ of vertex-disjoint bad $(D,M)$-matched $4$-cycles of a graph $G$, let $\overrightarrow{G}(S)$ be the digraph with vertex set $S$ and where $(A,B)$ is an arc in $\overrightarrow{G}(S)$ if $A$ is dependent on $B$. We remark that since $G$ is cubic, every vertex in $\overrightarrow{G}(S)$ has out-degree at most~$2$. Further by definition of a bad $(D,M)$-matched $4$-cycle, every vertex in $\overrightarrow{G}(S)$ has in-degree at most~$1$.

Given a rooted tree $T$ with root $r$, by an \emph{orientation of $T$} we mean orienting every arc of $T$ from a parent to its child.

\subsection{Proof of the main result}

We are now in a position to prove our main result, namely Theorem~\ref{thm:main}. Recall its statement.

\begin{unnumbered}{Theorem~\ref{thm:main}}
If $G$ is a twin-free, cubic  graph of order~$n$, then $\gL(G)\le \frac{n}{2}$.
\end{unnumbered}
\begin{proof}[\textbf{Proof of Theorem~\ref{thm:main}.}] Among all maximum matchings $M$ of $G$ and all sets $D \in \cD_G(M)$, we choose the matching $M_0$ and the set $D_0\in\cD_G(M_0)$ so that the number of $D_0$-bad pairs is minimum.
Let $X$ be a proper subset of vertices of $G$ such that $(|V|+|X|-\oc(G-X))/2$ is minimum. The structure of the graph $G$ with respect to the matching $M_0$ and the set $X$ is described in Theorem~\ref{t:structure}.

We now describe the structure of $D_0$-bad pairs:

\begin{clm}\label{clm:bad-pair-structure}
Every $D_0$-bad pair $\{u,v\}$ belongs to a common bad $(D_0,M_0)$-matched $4$-cycle, say $R$. Further, there exists a set $S_{u,v}$ of vertex-disjoint bad $(D_0,M_0)$-matched $4$-cycles containing $R$ such that the following holds: \1 \\
\indent {\rm (a)} For every $4$-cycle $C \in S_{u,v}$ and every vertex $x \in \{u'_C,v'_C\}$, either $x$ is adjacent to an \\ \hspace*{0.9cm} $M_0$-unmatched vertex in $G$, or $C$ is dependent on some other $C' \in S_{u,v}$ via the vertex $x$. \1 \\
\indent {\rm (b)}
$\overrightarrow{G}(S_{u,v})$ is an oriented tree rooted at $R$. \1 \\
\indent {\rm (c)} For every $4$-cycle $C \in S_{u,v}$, if both $u'_C$ and $v'_C$ have an $M_0$-unmatched neighbor, then these \\ \hspace*{0.9cm} neighbors are distinct and $\{u'_C,v'_C\} \subseteq X$.
\end{clm}
\noindent\emph{Proof of Claim~\ref{clm:bad-pair-structure}.} Let $\{u,v\}$ be a $D_0$-bad pair. Thus, $u$ and $v$ are vertices outside $D_0$ that are not located by $D_0$. By Lemma~\ref{lem:2dom}, both $u$ and $v$ are $2$-dominated by $D_0$, and by Lemma~\ref{lem:matched}, both $u$ and $v$ are $M_0$-matched. Let $u'$ and $v'$ be the $M_0$-matched neighbors of $u$ and~$v$, respectively. Since $u$ and $v$ are $2$-dominated by $D_0$, we note that $u'$ and $v'$ are the two common neighbors of $u$ and $v$ in $D_0$. Thus, $C_{uv} \colon uu'vv'u$ is a bad $(D_0,M_0)$-matched $4$-cycle in $G$. As observed earlier, every $4$-cycle in $G$ is an induced $4$-cycle. Hence, let $x,y,u''$ and $v''$ be the  neighbors of $u,v,u'$ and $v'$, respectively, that do not belong to this $4$-cycle $C_{uv}$. Let $D_u = (D_0\setminus \{u'\}) \cup \{u\}$ and let $D_v = (D_0\setminus \{v'\}) \cup \{v\}$. We proceed further with the following series of subclaims.

\begin{subclm}\label{subclm:twocases}
The following holds.\1 \\
\indent {\rm (a)}  If $D_u \notin \cD_G(M_0)$, then $u''$ is an $M_0$-unmatched vertex that is not adjacent to $u$.\1 \\
\indent {\rm (b)}  If $D_u \in \cD_G(M_0)$, then the only $D_u$-bad pair that is not a $D_0$-bad pair is $\{u'',z\}$ for some  \\  \hspace*{0.9cm} vertex~$z$. Moreover, $u''$ and $z$ are part of a bad $(D_0,M_0)$-matched $4$-cycle $C$ of $G$, and $C_{uv}$ \\  \hspace*{0.9cm} is dependent on $C$ via the vertex $u'$.
\end{subclm}
\noindent\emph{Proof of Claim~\ref{subclm:twocases}.} By definition of the sets in the family~$\cD_G(M_0)$, if $D_u \notin \cD_G(M_0)$, then $u''$ is an $M_0$-unmatched vertex and $u''$ is not adjacent to $u$, proving Statement~(a) of Claim~\ref{subclm:twocases}.

To prove Statement~(b), suppose that $D_u\in\cD_G(M_0)$. By our choice of $M_0$ and $D_0$, there are at least as many $D_u$-bad pairs as $D_0$-bad pairs. The only vertices that could potentially be negatively affected (in the sense that they are located by $D_0$ but not by $D_u$) by removing $u'$ from $D_0$ and replacing it with the vertex $u$ are $u'$, $u''$ and $v$. By Lemma~\ref{lem:2dom}, two vertices forming a $D_u$-bad pair are $2$-dominated by $D_u$. The vertex $v$ is $1$-dominated by $D_u$, and hence it is located by $D_u$.

Suppose that $u'$ is not located by $D_u$ from some other vertex outside $D_u$. Then, this vertex must be $x$, the neighbor of $u$ not on $C_{uv}$.  Considering the $D_u$-bad pair $\{u',x\}$, and noting that $u'$ and $x$ are $2$-dominated by $D_u$, we deduce that $u'' \in D_u$. If $x$ is $M_0$-unmatched, then by the definition of $\cD_G(M_0)$, we would have $u\in D_0$ and $u'\notin D_0$, a contradiction. Hence, $x$ is $M_0$-matched and its matched neighbor is in $D_u$. Since $x$ is $2$-dominated, we have $xu''\in M_0$. We now consider the maximum matching $M' = (M_0 \setminus \{uu',vv'\}) \cup \{uv',vu'\}$, and we let $D' = (D_0 \setminus \{u'\}) \cup \{v\}$.

We note that $u''x \in M'$. Since neither $u$ nor $u'$ has an $M'$-unmatched neighbor, we note that $D' \in \cD(M')$. The only vertices that could potentially be negatively affected (in the sense that they are located by $D_0$ but not by $D'$) by these changes are the vertices dominated by $u'$ in $D_0$ and that do not belong to $D'$. The only such vertices are $u$ and $u'$. By Lemma~\ref{lem:2dom}, if two vertices form a $D'$-bad pair, then they are $2$-dominated by $D'$. The vertex $u$ is $1$-dominated by $D'$, and hence it is located by $D'$. Thus, the vertex $u'$ is not located by $D'$ from some other vertex outside $D'$. Such a vertex must be adjacent to both $v$ and $u''$, and is therefore the neighbor of $v$ outside $C_{uv}$, namely the vertex~$y$.  If $x = y$, then $u$ and $v$ would be open twins, a contradiction. Hence, $x \ne y$. If $y$ is $M'$-matched, since $u''x\in M'$ and $u'v\in M'$, the vertex $y$ is $3$-dominated by $D'$, a contradiction. Hence, $y$ is $M'$-unmatched (and $M$-unmatched). Then, since $D_0\in\cD_G(M_0)$, the vertices $y$ and $v'$ are adjacent. Hence, $y$ is $3$-dominated by $D'$, a contradiction. Thus, $u'$ is located by $D_u$.

Therefore, among the vertices dominated by $u'$ in $D_0$, the vertex $u''$ is the only vertex that was located by $D_0$ but that is not located by $D_u$ from some other vertex, $z$ say, outside $D_u$. Thus, $\{u'',z\}$ is the only pair of vertices located by $D_0$ but not by $D_u$. Hence, the number of $D_u$-bad pairs is the same as the number of $D_0$-bad pairs (since $\{u,v\}$ is not a $D_u$-bad pair) and we can apply Lemma~\ref{lem:matched} to $D_u$ to deduce that $u''$ and $z$ are $M_0$-matched vertices.  By Lemma~\ref{lem:2dom}, both $u''$ and $z$ are $2$-dominated by $D_u$. Let $w$ and $t$ be the $M_0$-matched neighbors of $u''$ and~$z$, respectively. Since $u''$ and $z$ are $2$-dominated by $D_u$ and $D_u \in \cD_G(M_0)$, we note that $w$ and $t$ are the two common neighbors of $u''$ and $z$ in $D_u$. Thus, the $4$-cycle $C \colon u''wztu''$ is a bad $(D_0,M_0)$-matched $4$-cycle in $G$ with $u'' = u_C$, $w = u'_C$, $z = v_C$ and $t = v'_C$, and $C_{uv}$ is dependent on $C$ via the vertex $u'$. This establishes Statement~(b) of Claim~\ref{subclm:twocases}.~\smallqed

\medskip
Interchanging the roles of $u$ and $v$ in the proof of Claim~\ref{subclm:twocases}, we have the following analogous result for the vertex~$v$.

\begin{subclm}\label{subclm:twocases-v}
The following holds.\1 \\
\indent {\rm (a)} If $D_v \notin \cD_G(M_0)$, then $v''$ is an $M_0$-unmatched vertex that is not adjacent to $v$. \1 \\
\indent {\rm (b)} If $D_v \in \cD_G(M_0)$, then the only $D_v$-bad pair that is not a $D_0$-bad pair is $\{v'',z\}$ for some \\  \hspace*{0.9cm} vertex~$z$. Moreover, $v''$ and $z$ are part of a bad $(D_0,M_0)$-matched $4$-cycle $C$ of $G$, and $C_{uv}$ \\  \hspace*{0.9cm} is dependent on $C$ via the vertex $v'$.
\end{subclm}

Let $R$ denote the bad $(D_0,M_0)$-matched $4$-cycle $C_{uv} \colon uu'vv'u$, where $u = u_R$, $u' = u'_R$, $v = v_R$ and $v' = v'_R$. We now show the existence of a set $S_{u,v}$ of vertex-disjoint bad $(D_0,M_0)$-matched $4$-cycles containing $R$ such that conditions~(a) and~(b) in the statement of Claim~\ref{clm:bad-pair-structure} hold. If both $u'$ and $v'$ have an $M_0$-unmatched neighbor, then we let $S_{u,v} = \{R\}$ and we are done.

Otherwise, renaming vertices if necessary, we may assume, by Claim~\ref{subclm:twocases} and Claim~\ref{subclm:twocases-v}, that $D_u \in \cD_G(M_0)$ and that $R$ is dependent on a bad $(D_0,M_0)$-matched $4$-cycle $C$ via vertex $u'_R$.
Now, since by Claim~\ref{subclm:twocases}(b) the number of $D_u$-bad pairs is the same as the number of $D_0$-bad pairs (hence $D_u$ also minimizes the number of bad pairs), we can apply Claim~\ref{subclm:twocases} and Claim~\ref{subclm:twocases-v} to $C$, $D_u$ and to the $D_u$-bad pair $\{u_C,v_C\}$. This shows that each of $u'_C$ and $v'_C$ either have an $M_0$-unmatched neighbor, or $C$ is dependent on some other bad $(D_0,M_0)$-matched $4$-cycle via this vertex. Repeating this process as long as possible yields a set $S_{u,v}$ of bad $(D_0,M_0)$-matched $4$-cycles, where for each bad $(D_0,M_0)$-matched $4$-cycle in $S_{u,v}$ different from $R$ there is some other bad $(D_0,M_0)$-matched $4$-cycle in $S_{u,v}$ that depends on it, and satisfies the properties in Claim~\ref{subclm:twocases} and Claim~\ref{subclm:twocases-v}. This establishes Statement~(a) of Claim~\ref{clm:bad-pair-structure}. Moreover we have the following.

\begin{subclm}\label{subclm:vertex-disjoint}
Any two distinct $4$-cycles in $S_{u,v}$ are vertex-disjoint.
\end{subclm}
\noindent\emph{Proof of Claim~\ref{subclm:vertex-disjoint}.}
Let $A:u_Au'_Av_Av'_A$ and $B:u_Bu'_Bv_Bv'_B$ be two distinct $4$-cycles of $S_{u,v}$. If they have a common vertex, since their vertices are pairwise $M_0$-matched, they must have two vertices in common, and these vertices must be $M_0$-matched to each other. But then, the vertex that belongs to both $A$ and $B$ but does not belong to $D_0$ must be $3$-dominated by $D_0$, a contradiction.~\smallqed

\medskip
Now, consider the digraph $\overrightarrow{G}(S_{u,v})$, which by Claim~\ref{subclm:vertex-disjoint} is well-defined. The following properties hold in the digraph $\overrightarrow{G}(S_{u,v})$. Recall that the distance from a vertex $x$ to a vertex $y$ in a directed graph $D$ is the minimum length among all directed paths from $x$ to $y$ in $D$.

\begin{subclm}\label{subclm:oriented-tree}
The following holds. \1 \\
\indent {\rm (a)} $R$ has in-degree~$0$ in $\overrightarrow{G}(S_{u,v})$.\\
\indent {\rm (b)} Every vertex in $\overrightarrow{G}(S_{u,v})$ different from $R$ has in-degree exactly~$1$. \\
\indent {\rm (c)} $\overrightarrow{G}(S_{u,v})$ has no directed cycle.
\end{subclm}
\noindent\emph{Proof of Claim~\ref{subclm:oriented-tree}.} To see that Statement~(a) holds, observe that if some bad $(D_0,M_0)$-matched $4$-cycle $C$ was dependent on $R$ say, on vertex $u_R$, then $u_R$ would be $3$-dominated by $D_0$, a contradiction.

For Statement~(b), we show firstly that $\overrightarrow{G}(S_{u,v})$ has maximum in-degree~$1$. Suppose to the contrary that for some bad $(D_0,M_0)$-matched $4$-cycle $A$ in $S_{u,v}$, there are two other bad $(D_0,M_0)$-matched $4$-cycles $B$ and $C$ of $S_{u,v}$ that are both dependent on $A$. Since $B$ is dependent on $A$, the vertex $u_A$ is adjacent to $u_B'$ or $v_B'$. Since $C$ is dependent on $A$, the vertex $u_A$ is adjacent to $u_C'$ or $v_C'$. Thus, the vertex $u_A$ has degree at least~$4$, a contradiction. We show secondly that $R$ is the only vertex in $\overrightarrow{G}(S_{u,v})$ with in-degree~$0$. As observed in the paragraph immediately preceding Claim~A.3, for each bad $(D_0,M_0)$-matched $4$-cycle in $S_{u,v}$ different from $R$ there is some other bad $(D_0,M_0)$-matched $4$-cycle in $S_{u,v}$ that depends on it. Therefore, every bad $(D_0,M_0)$-matched $4$-cycle in $S_{u,v}$ different from $R$ has in-degree at least~$1$ in $\overrightarrow{G}(S_{u,v})$. Therefore, by our earlier observations, every bad $(D_0,M_0)$-matched $4$-cycle in $S_{u,v}$ different from $R$ has in-degree exactly~$1$ in $\overrightarrow{G}(S_{u,v})$.

For Statement~(c), suppose to the contrary that $\overrightarrow{G}(S_{u,v})$ contains a directed cycle $C \colon C_1 C_2 \ldots C_kC_1$ for some $k \ge 2$. By Statement~(a), we know that $R$ has in-degree~0 and therefore cannot belong to this cycle. However, there is a (directed) path from $R$ to every other vertex in $\overrightarrow{G}(S_{u,v})$. Among all vertices in the directed cycle $C$, let $C_i$ be chosen so that the distance from $R$ to $C_i$ in $\overrightarrow{G}(S_{u,v})$ is minimum where $i \in [k]$. Let $P$ be a shortest (directed) path from $R$ to $C_i$ in $\overrightarrow{G}(S_{u,v})$ and let $B$ be the vertex that immediately precedes $C_i$ on the path $P$ (possibly, $B = R$). Since the distance from $R$ to $B$ in $\overrightarrow{G}(S_{u,v})$ is less than the distance from $R$ to $C_i$ in $\overrightarrow{G}(S_{u,v})$, the vertex $B$ does not belong to the directed cycle $C$. Therefore, $C_i$ has in-degree at least~$2$ in $\overrightarrow{G}(S_{u,v})$, contradicting Statement~(b). This completes the proof of the claim.~\smallqed

\medskip
By Claim~\ref{subclm:oriented-tree}(c), if there is a cycle in $\overrightarrow{G}(S_{u,v})$, it cannot be an oriented cycle. But then some vertex in that cycle must have in-degree at least~2, contradicting Claim~\ref{subclm:oriented-tree}(b). Hence, Claim~\ref{subclm:oriented-tree} implies that $\overrightarrow{G}(S_{u,v})$ is an oriented tree rooted at $R$, and we have proved Statement~(b) of Claim~\ref{clm:bad-pair-structure}.

It remains to prove Statement~(c) of Claim~\ref{clm:bad-pair-structure}. Let $C \in S_{u,v}$ be a bad $(D_0,M_0)$-matched $4$-cycle where both $u'_C$ and $v'_C$ have an $M_0$-unmatched neighbor, $u''$ and $v''$ say, respectively. These neighbors are clearly distinct, since otherwise $u'_C$ and $v'_C$ are open twins. By Theorem~\ref{t:structure}, the two $M_0$-unmatched vertices $u''$ and $v''$ belong to distinct odd components of $G-X$. Suppose to the contrary that $u_C \in X$. Then, by Theorem~\ref{t:structure}, $u'_C \notin X$ and $u_C'$ belongs to an odd component of $G - X$ that contains no $M_0$-unmatched vertex. However, $u_C'$ is adjacent to the $M_0$-unmatched vertex $u''$ which implies that $u_C'$ belongs to the same odd component of $G - X$ as the $M_0$-unmatched vertex $u''$, a contradiction. Hence, $u_C \notin X$. Analogously, $v_C \notin X$.

If neither $u'_C$ nor $v'_C$ belong to $X$, then $u''$ and $v''$ belong to the same components of $G-X$, a contradiction. Hence, renaming $u'_C$ and $v'_C$, if necessary, we may assume that $v_C' \in X$. Thus, $v_C$ belongs to an odd component of $G - X$ that contains no $M_0$-unmatched vertex. If $u_C' \notin X$, then $v_C$ would be in the same odd component of $G - X$ as the $M_0$-unmatched vertex $u''$, a contradiction. Hence, $u_C' \in X$. This establishes Statement~(c) of Claim~\ref{clm:bad-pair-structure} and completes the proof of Claim~\ref{clm:bad-pair-structure}.~$(\Box)$

An example of a subgraph of $G$ corresponding to a set $S_{u,v}$ that contains six bad $(M_0,D_0)$-matched $4$-cycles is illustrated in Figure~\ref{fig:tree}(a) (where the edges of $M_0$ are thickened, black vertices belong to $D_0$ and white vertices do not, and square white vertices are $M_0$-unmatched vertices).

\begin{figure}[ht]
  \centering
  \subfigure[The structure around the $D_0$-bad pair.]{
    \scalebox{0.8}{\begin{tikzpicture}[join=bevel,inner sep=0.6mm,line width=0.8pt, scale=0.4]

        \path (0,0) node[draw, shape=circle] (uR) {};
        \draw (uR) node[below=0.1cm] {$u_R$};
        \path (uR)+(-1.5,0) node[draw, shape=circle] (yR) {};
        \path (uR)+(2,0) node[draw, shape=circle] (vR) {};
        \draw (vR) node[below=0.1cm] {$v_R$};
        \path (vR)+(1.5,0) node[draw, shape=circle] (xR) {};
        \path (uR)+(0,2) node[draw, fill, shape=circle] (u'R) {};
        \draw (u'R) node[left=0.1cm] {$u'_R$};
        \path (vR)+(0,2) node[draw, fill, shape=circle] (v'R) {};
        \draw (v'R) node[right=0.1cm] {$v'_R$};
        \draw[line width=0.5pt] (uR)--(v'R) (xR)--(vR)--(u'R) (uR)--(yR) ;
        \draw[line width=2pt] (uR)--(u'R) (vR)--(v'R);

        \path (uR)+(-2,6) node[draw, shape=circle] (uC1) {};
        \draw (uC1) node[right=0.1cm] {$u_{C_1}$};
        \path (uC1)+(-2,0) node[draw, shape=circle] (vC1) {};
        \draw (vC1) node[below=0.1cm] {$v_{C_1}$};
        \path (uC1)+(0,2) node[draw, fill, shape=circle] (u'C1) {};
        \draw (u'C1) node[right=0.1cm] {$u'_{C_1}$};
        \path (vC1)+(0,2) node[draw, fill, shape=circle] (v'C1) {};
        \draw (v'C1) node[left=0.1cm] {$v'_{C_1}$};
        \path (vC1)+(-1.5,0) node[draw, shape=circle] (xC1) {};
        \draw[line width=0.5pt] (uC1)--(v'C1) (xC1)--(vC1)--(u'C1);
        \draw[line width=2pt] (uC1)--(u'C1) (vC1)--(v'C1);

        \path (uR)+(3,5) node[draw, shape=circle] (uC2) {};
        \draw (uC2) node[left=0.1cm] {$u_{C_2}$};
        \path (uC2)+(2,0) node[draw, shape=circle] (vC2) {};
        \draw (vC2) node[below=0.1cm] {$v_{C_2}$};
        \path (uC2)+(0,2) node[draw, fill, shape=circle] (u'C2) {};
        \draw (u'C2) node[left=0.1cm] {$u'_{C_2}$};
        \path (vC2)+(0,2) node[draw, fill, shape=circle] (v'C2) {};
        \draw (v'C2) node[right=0.1cm] {$v'_{C_2}$};
        \path (u'C2)+(-0.5,1.5) node[draw, shape=rectangle, minimum size=5pt] (u''C2) {};
        \path (v'C2)+(0.5,1.5) node[draw, shape=rectangle, minimum size=5pt] (v''C2) {};
        \path (vC2)+(1.5,0) node[draw, shape=circle] (xC2) {};
        \draw[line width=0.5pt] (uC2)--(v'C2)--(v''C2) (xC2)--(vC2)--(u'C2)--(u''C2);
        \draw[line width=2pt] (uC2)--(u'C2) (vC2)--(v'C2);

        \path (vC1)+(2,6) node[draw, shape=circle] (uC3) {};
        \draw (uC3) node[right=0.1cm] {$u_{C_3}$};
        \path (uC3)+(-2,0) node[draw, shape=circle] (vC3) {};
        \draw (vC3) node[below=0.1cm] {$v_{C_3}$};
        \path (uC3)+(0,2) node[draw, fill, shape=circle] (u'C3) {};
        \draw (u'C3) node[right=0.1cm] {$u'_{C_3}$};
        \path (vC3)+(0,2) node[draw, fill, shape=circle] (v'C3) {};
        \draw (v'C3) node[left=0.1cm] {$v'_{C_3}$};
        \path (vC3)+(-1.5,0) node[draw, shape=circle] (xC3) {};
        \path (v'C3)+(-0.5,1.5) node[draw, shape=rectangle, minimum size=5pt] (v''C3) {};
        \draw[line width=0.5pt] (uC3)--(v'C3)--(v''C3) (xC3)--(vC3)--(u'C3);
        \draw[line width=2pt] (uC3)--(u'C3) (vC3)--(v'C3);

        \path (vC1)+(7,6) node[draw, shape=circle] (uC4) {};
        \draw (uC4) node[left=0.1cm] {$u_{C_4}$};
        \path (uC4)+(2,0) node[draw, shape=circle] (vC4) {};
        \draw (vC4) node[below=0.1cm] {$v_{C_4}$};
        \path (uC4)+(0,2) node[draw, fill, shape=circle] (u'C4) {};
        \draw (u'C4) node[left=0.1cm] {$u'_{C_4}$};
        \path (vC4)+(0,2) node[draw, fill, shape=circle] (v'C4) {};
        \draw (v'C4) node[right=0.1cm] {$v'_{C_4}$};
        \path (u'C4)+(-0.5,1.5) node[draw, shape=rectangle, minimum size=5pt] (u''C4) {};
        \path (vC4)+(1.5,0) node[draw, shape=circle] (xC4) {};
        \draw[line width=0.5pt] (uC4)--(v'C4) (xC4)--(vC4)--(u'C4)--(u''C4);
        \draw[line width=2pt] (uC4)--(u'C4) (vC4)--(v'C4);

        \path (uC3)+(2,6) node[draw, shape=circle] (uC5) {};
        \draw (uC5) node[right=0.1cm] {$u_{C_5}$};
        \path (uC5)+(-2,0) node[draw, shape=circle] (vC5) {};
        \draw (vC5) node[below=0.1cm] {$v_{C_5}$};
        \path (uC5)+(0,2) node[draw, fill, shape=circle] (u'C5) {};
        \draw (u'C5) node[right=0.1cm] {$u'_{C_5}$};
        \path (vC5)+(0,2) node[draw, fill, shape=circle] (v'C5) {};
        \draw (v'C5) node[left=0.1cm] {$v'_{C_5}$};
        \path (vC5)+(-1.5,0) node[draw, shape=circle] (xC5) {};
        \path (v'C5)+(-0.5,1.5) node[draw, shape=rectangle, minimum size=5pt] (v''C5) {};
        \path (u'C5)+(0.5,1.5) node[draw, shape=rectangle, minimum size=5pt] (u''C5) {};
        \draw[line width=0.5pt] (uC5)--(v'C5)--(v''C5) (xC5)--(vC5)--(u'C5)--(u''C5);
        \draw[line width=2pt] (uC5)--(u'C5) (vC5)--(v'C5);


        \draw[line width=0.5pt] (u'R)--(uC1) (v'R)--(uC2) (v'C1)--(uC3) (u'C1)--(uC4) (u'C3)--(uC5) (v'C4) .. controls +(0,4) and +(4,0) .. (u''C5);

    \end{tikzpicture}}}
 \qquad
  \subfigure[The oriented tree $\protect\overrightarrow{G}(S_{u_R,v_R})$.]{
    \scalebox{0.8}{\begin{tikzpicture}[join=bevel,inner sep=0.6mm,line width=0.8pt, scale=0.4]

        \path (0,0) node (a) {~};
        \path (0,2) node[draw,rectangle, rounded corners,line width=1.2pt, minimum size=1cm] (R) {$R$};
        \path (R)+(-2,6) node[draw,rectangle, rounded corners,line width=1.2pt, minimum size=1cm] (C1) {$C_1$};
        \path (R)+(3,5) node[draw,rectangle, rounded corners,line width=1.2pt, minimum size=1cm] (C2) {$C_2$};
        \path (C1)+(0,6) node[draw,rectangle, rounded corners,line width=1.2pt, minimum size=1cm] (C3) {$C_3$};
        \path (C1)+(5,6) node[draw,rectangle, rounded corners,line width=1.2pt, minimum size=1cm] (C4) {$C_4$};
        \path (C3)+(2,6) node[draw,rectangle, rounded corners,line width=1.2pt, minimum size=1cm] (C5) {$C_5$};

        \draw[line width=1.2pt, ->] (R)--(C1);
        \draw[line width=1.2pt, ->] (C1)--(C3);
        \draw[line width=1.2pt, ->] (C3)--(C5);
        \draw[line width=1.2pt, ->] (R)--(C2);
        \draw[line width=1.2pt, ->] (C1)--(C4);
    \end{tikzpicture}}}
\qquad
  \subfigure[The modification of $D_0$ around $\{u_R,v_R\}$  when associating it with the leaf $C_5$ of $\protect\overrightarrow{G}(S_{u_R,v_R})$. The vertices of each circled pair get swapped in $D_0$.]{
    \scalebox{0.8}{\begin{tikzpicture}[join=bevel,inner sep=0.6mm,line width=0.8pt, scale=0.4]

        \path (0,0) node[draw, fill, shape=circle] (uR) {};
        \draw (uR) node[below right=0.12cm] {$u_R$};
        \path (uR)+(-1.5,0) node[draw, shape=circle] (yR) {};
        \path (uR)+(2,0) node[draw, shape=circle] (vR) {};
        \draw (vR) node[below right=0.05cm] {$v_R$};
        \path (vR)+(1.5,0) node[draw, shape=circle] (xR) {};
        \path (uR)+(0,2) node[draw, shape=circle] (u'R) {};
        \draw (u'R) node[left=0.2cm] {$u'_R$};
        \path (vR)+(0,2) node[draw, fill, shape=circle] (v'R) {};
        \draw (v'R) node[right=0.1cm] {$v'_R$};
        \draw[line width=0.5pt] (uR)--(v'R) (xR)--(vR)--(u'R) (uR)--(yR) ;
        \draw[line width=2pt] (uR)--(u'R) (vR)--(v'R);
        \draw[line width=1pt] (uR)+(0,1) ellipse(0.5cm and 2cm) node{};

        \path (uR)+(-2,6) node[draw, shape=circle] (uC1) {};
        \draw (uC1) node[right=0.1cm] {$u_{C_1}$};
        \path (uC1)+(-2,0) node[draw, fill, shape=circle] (vC1) {};
        \draw (vC1) node[below right=0.12cm] {$v_{C_1}$};
        \path (uC1)+(0,2) node[draw, fill, shape=circle] (u'C1) {};
        \draw (u'C1) node[right=0.1cm] {$u'_{C_1}$};
        \path (vC1)+(0,2) node[draw, shape=circle] (v'C1) {};
        \draw (v'C1) node[left=0.2cm] {$v'_{C_1}$};
        \path (vC1)+(-1.5,0) node[draw, shape=circle] (xC1) {};
        \draw[line width=0.5pt] (uC1)--(v'C1) (xC1)--(vC1)--(u'C1);
        \draw[line width=2pt] (uC1)--(u'C1) (vC1)--(v'C1);
        \draw[line width=1pt] (vC1)+(0,1) ellipse(0.5cm and 2cm) node{};

        \path (uR)+(3,5) node[draw, shape=circle] (uC2) {};
        \draw (uC2) node[left=0.1cm] {$u_{C_2}$};
        \path (uC2)+(2,0) node[draw, shape=circle] (vC2) {};
        \draw (vC2) node[below=0.1cm] {$v_{C_2}$};
        \path (uC2)+(0,2) node[draw, fill, shape=circle] (u'C2) {};
        \draw (u'C2) node[left=0.1cm] {$u'_{C_2}$};
        \path (vC2)+(0,2) node[draw, fill, shape=circle] (v'C2) {};
        \draw (v'C2) node[right=0.1cm] {$v'_{C_2}$};
        \path (u'C2)+(-0.5,1.5) node[draw, shape=rectangle, minimum size=5pt] (u''C2) {};
        \path (v'C2)+(0.5,1.5) node[draw, shape=rectangle, minimum size=5pt] (v''C2) {};
        \path (vC2)+(1.5,0) node[draw, shape=circle] (xC2) {};
        \draw[line width=0.5pt] (uC2)--(v'C2)--(v''C2) (xC2)--(vC2)--(u'C2)--(u''C2);
        \draw[line width=2pt] (uC2)--(u'C2) (vC2)--(v'C2);

        \path (vC1)+(2,6) node[draw, fill, shape=circle] (uC3) {};
        \draw (uC3) node[right=0.2cm] {$u_{C_3}$};
        \path (uC3)+(-2,0) node[draw, shape=circle] (vC3) {};
        \draw (vC3) node[below=0.1cm] {$v_{C_3}$};
        \path (uC3)+(0,2) node[draw, shape=circle] (u'C3) {};
        \draw (u'C3) node[right=0.2cm] {$u'_{C_3}$};
        \path (vC3)+(0,2) node[draw, fill, shape=circle] (v'C3) {};
        \draw (v'C3) node[left=0.1cm] {$v'_{C_3}$};
        \path (vC3)+(-1.5,0) node[draw, shape=circle] (xC3) {};
        \path (v'C3)+(-0.5,1.5) node[draw, shape=rectangle, minimum size=5pt] (v''C3) {};
        \draw[line width=0.5pt] (uC3)--(v'C3)--(v''C3) (xC3)--(vC3)--(u'C3);
        \draw[line width=2pt] (uC3)--(u'C3) (vC3)--(v'C3);
        \draw[line width=1pt] (uC3)+(0,1) ellipse(0.5cm and 2cm) node{};

        \path (vC1)+(7,6) node[draw, shape=circle] (uC4) {};
        \draw (uC4) node[left=0.1cm] {$u_{C_4}$};
        \path (uC4)+(2,0) node[draw, shape=circle] (vC4) {};
        \draw (vC4) node[below=0.1cm] {$v_{C_4}$};
        \path (uC4)+(0,2) node[draw, fill, shape=circle] (u'C4) {};
        \draw (u'C4) node[left=0.1cm] {$u'_{C_4}$};
        \path (vC4)+(0,2) node[draw, fill, shape=circle] (v'C4) {};
        \draw (v'C4) node[right=0.1cm] {$v'_{C_4}$};
        \path (u'C4)+(-0.5,1.5) node[draw, shape=rectangle, minimum size=5pt] (u''C4) {};
        \path (vC4)+(1.5,0) node[draw, shape=circle] (xC4) {};
        \draw[line width=0.5pt] (uC4)--(v'C4) (xC4)--(vC4)--(u'C4)--(u''C4) (v'C4) .. controls +(0,4) and +(4,0) .. (u''C5);
        \draw[line width=2pt] (uC4)--(u'C4) (vC4)--(v'C4);

        \path (uC3)+(2,6) node[draw, fill, shape=circle] (uC5) {};
        \draw (uC5) node[right=0.2cm] {$u_{C_5}$};
        \path (uC5)+(-2,0) node[draw, shape=circle] (vC5) {};
        \draw (vC5) node[below=0.1cm] {$v_{C_5}$};
        \path (uC5)+(0,2) node[draw, shape=circle] (u'C5) {};
        \draw (u'C5) node[right=0.2cm] {$u'_{C_5}$};
        \path (vC5)+(0,2) node[draw, fill, shape=circle] (v'C5) {};
        \draw (v'C5) node[left=0.1cm] {$v'_{C_5}$};
        \path (vC5)+(-1.5,0) node[draw, shape=circle] (xC5) {};
        \path (v'C5)+(-0.5,1.5) node[draw, shape=rectangle, minimum size=5pt] (v''C5) {};
        \path (u'C5)+(0.5,1.5) node[draw, shape=rectangle, minimum size=5pt] (u''C5) {};
        \draw[line width=0.5pt] (uC5)--(v'C5)--(v''C5) (xC5)--(vC5)--(u'C5)--(u''C5);
        \draw[line width=2pt] (uC5)--(u'C5) (vC5)--(v'C5);
        \draw[line width=1pt] (uC5)+(0,1) ellipse(0.5cm and 2cm) node{};

        \draw[line width=0.5pt] (u'R)--(uC1) (v'R)--(uC2) (v'C1)--(uC3) (u'C1)--(uC4) (u'C3)--(uC5);

    \end{tikzpicture}}}

  \caption{Example of a $D_0$-bad pair $\{u_R,v_R\}$ with the set of
    bad $(D_0,M_0)$-matched $4$-cycles $S_{u_R,v_R}=\{R,C_1,\ldots,C_5\}$. The edges of $M_0$
    are thickened; squared vertices are $M_0$-unmatched; black vertices belong to $D_0$.}
  \label{fig:tree}
\end{figure}
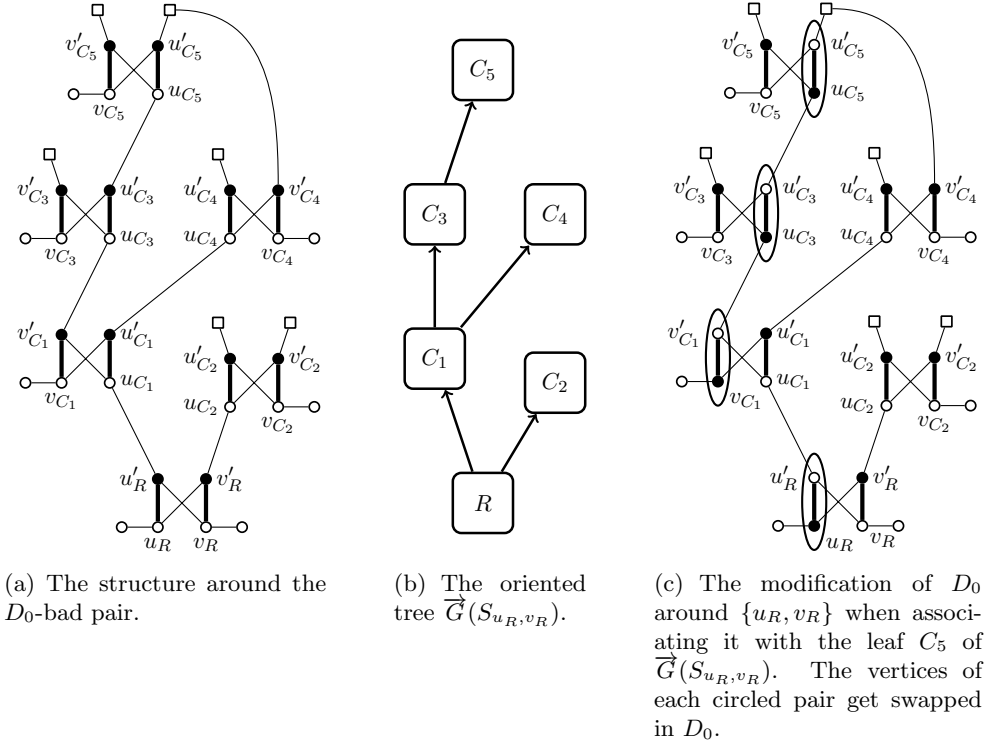

We now return to the proof of Theorem~\ref{t:main}. Our strategy is to modify the set $D_0$ in such a way that the resulting set  becomes a locating-dominating set of $G$ of cardinality at most one-half the order of~$G$. Consider the set of $D_0$-bad pairs. By Claim~\ref{clm:bad-pair-structure},
each such $D_0$-bad pair $\{u,v\}$ belongs to a bad $(D_0,M_0)$-matched $4$-cycle $R$ and there
is a set $S_{u,v}$ of vertex-disjoint bad $(D_0,M_0)$-matched $4$-cycles such that
$\overrightarrow{G}(S_{u,v})$ is an oriented tree rooted at $R$. We note that, given two $D_0$-bad pairs $\{u,v\}$ and $\{x,y\}$, the trees $\overrightarrow{G}(S_{u,v})$ and $\overrightarrow{G}(S_{x,y})$ are vertex-disjoint, and furthermore no bad $(D_0,M_0)$-matched $4$-cycles of $\overrightarrow{G}(S_{u,v})$ and $\overrightarrow{G}(S_{x,y})$ share any vertex. Indeed, by similar arguments as in Claim~\ref{subclm:vertex-disjoint} and Claim~\ref{subclm:oriented-tree}, we would otherwise contradict the definition of a bad $(D_0,M_0)$-matched $4$-cycle.

Now, given a $D_0$-bad pair $\{u,v\}$, consider
any leaf $C$ of $\overrightarrow{G}(S_{u,v})$. By Claim~\ref{clm:bad-pair-structure}(a), the vertices
$u'_C$ and $v'_C$ both have a distinct $M_0$-unmatched neighbor, say
$u_C''$ and $v_C''$, respectively. Further, by Claim~\ref{clm:bad-pair-structure}(c), both $u'_C$ and $v'_C$ belong to $X$.

For each $D_0$-bad pair $\{u,v\}$, we select an arbitrary leaf $C$ of $\overrightarrow{G}(S_{u,v})$ and associate the pair of vertices $u''=u_C''$ and $v''=v_C''$ of $M_0$-unmatched neighbors of $u'_C$ and $v'_C$, respectively, with the pair $\{u,v\}$, and we write $f(u,v)=\{u'',v''\}$. Let $V^*$ be the set of all $M_0$-unmatched vertices associated with some $D_0$-bad pair. We define the (multi)graph
$G^*$ on the vertex set $V^*$ by adding an edge joining $u''$ and $v''$ for each
$D_0$-bad pair $\{u,v\}$ such that $f(u,v) = \{u'',v''\}$.  As remarked earlier, the vertices  $u''$ and $v''$ are distinct, implying that $G^*$ has no loops (although it may have
multiple edges), no isolated vertices, and is subcubic (that is, has maximum degree at most~$3$).
Our aim is to add at most $|V^*|/2$ vertices to $D_0$ and to locally modify $D_0$ around the $D_0$-bad pairs in order to obtain a locating-dominating set, $D'$, of cardinality
\[
|D'| \le \match(G)+\frac{|V^*|}{2} \leq \match(G)+\frac{n-2\match(G)}{2} = \frac{n}{2}.
\]

We now describe the construction of such a set $D'$. Let $D^*$ be a minimum dominating set of $G^*$. Since $G^*$ has no isolated vertex, $|D^*| \le |V^*|/2$. Since $G^*$ has maximum degree at most~$3$ and since every vertex outside $D^*$ is adjacent to at least one vertex of $D^*$ in $G^*$, we note that $G^* - D^*$ has maximum degree at most~$2$. We now build a locating-dominating set from $D_0$ by adding $D^*$ to $D_0$ and by propagating modifications of $D_0$ along the oriented trees associated with all $D_0$-bad pairs. More precisely, we perform our propagation as follows.

\medskip
\noindent {\bf Step~1: We first consider all \boldmath{$D_0$}-bad pairs associated with a pair of vertices of \boldmath{$G^*$} at least one vertex of which belongs to the set \boldmath{$D^*$}.}  Let $\{u,v\}$ be such a $D_0$-bad pair, and let $u''$ and $v''$ be the vertices of $V^*$ such that $f(u,v) = \{u'',v''\}$. Adopting our earlier notation, let $u''=u''_C$ and $v''=v''_C$, where $C$ is the chosen leaf in the tree $\overrightarrow{G}(S_{u,v})$. Let $R$ be the bad $(D_0,M_0)$-matched $4$-cycle in $S_{u,v}$ containing $u$ and $v$. Renaming $u''$ and $v''$, if necessary, we may assume that $u''$ belongs to $D^*$. We now consider the unique (directed) path $P$ of $\overrightarrow{G}(S_{u,v})$ joining $R$ to $C$ and we modify $D_0$ along $P$ as follows. First, replace $u'_C$ with $u_C$ in $D_0$. If $B$ is the parent of $C$ in $\overrightarrow{G}(S_{u,v})$ (and so, $B$ is the vertex on the $(R,C)$-path
$P$ that immediately precedes~$C$) and $B$ is dependent on $C$ via $x'_B$, where $x_B\in\{u_B,v_B\}$, we replace $x'_B$ with $x_B$ in $D_0$. We continue this process until we perform the modification in the root $R$.
This exchange argument in the oriented tree $\overrightarrow{G}(S_{u,v})$ associated with the subgraph of $G$ corresponding to the set $S_{u,v}$ illustrated in Figure~\ref{fig:tree}(a) is shown in Figure~\ref{fig:tree}(c).
This process is done for all $D_0$-bad pairs associated with a pair of vertices of $G^*$ with at least one member in $D^*$. Let $D'$ be the resulting modified set $D_0$.

\begin{clm}\label{clm:step1}
The set of $(D' \cup D^*)$-bad pairs is a proper subset of the set of $D_0$-bad pairs.
\end{clm}
\noindent\emph{Proof of Claim~\ref{clm:step1}.} Let $\{u,v\}$ be an original $D_0$-bad pair associated with a pair of vertices of $G^*$ at least one of which belongs to the set $D^*$. Since at least one of $u$ and $v$ now belongs to $D'$, the pair $\{u,v\}$ is not a $(D'\cup D^*)$-bad pair.
It suffices to check the pairs of vertices that could possibly have been affected by the exchange arguments; that is, all vertices previously dominated by a vertex that has been removed from $D_0$ to construct $D'$ (this includes all vertices removed from $D_0$ to construct $D'$). %
A vertex affected by the modification belongs to a bad $(D_0,M_0)$-matched $4$-cycle $A$ in the selected path of $\overrightarrow{G}(S_{u,v})$ for some $D_0$-bad pair $\{u,v\}$ that is associated with a pair of vertices of $G^*$ at least one of which belongs to the set $D^*$.
Let the vertex set of $A$ be $\{x_A,y_A,x'_A,y'_A\}$ with $\{x_A,y_A\}=\{u_A,v_A\}$, where $x'_A$ has been replaced with $x_A$ in $D'$.

It is sufficient to check that the vertices $x'_A$, $y_A$ and the neighbor, $z$ say, of $x'_A$ not in $A$ are located by $D' \cup D^*$ or belong to $D' \cup D^*$.
We observe that even though $D'$ might not belong to $\cD(M)$, the set $D'$ contains exactly one vertex from each edge in $M_0$.
We note that every vertex that was removed from $D_0$ during the exchange arguments when constructing $D'$ is either adjacent to a vertex of $D^*$ or is adjacent to no $M_0$-unmatched vertex. Hence, the vertices of $D_0$ that are adjacent to an $M_0$-unmatched vertex that does not belong to $D^*$ are not removed from $D_0$ during the exchange arguments, implying by Lemma~\ref{lem:M-unmatched-2dom} that every $M_0$-unmatched vertex is adjacent to at least two vertices in $D' \cup D^*$ or belongs to $D^*$.
It follows that every vertex that is $1$-dominated by $D'\cup D^*$ is located by this set.
In particular, irrespective of whether $y=u$ or $y=v$, the vertex $y_A$ is $1$-dominated by
$D'\cup D^*$ and is thus located by
$D'\cup D^*$.

Suppose that $z$ is an $M_0$-unmatched vertex. Then, by construction, $z \in D^*$. In this case, $x'_A$ is dominated by $z$ and $x_A$ but by no other vertex in $D' \cup D^*$. If another vertex $w$ is also only dominated by $z$ and $x_A$ from $D' \cup D^*$, then such a vertex cannot be $M_0$-unmatched because the set of $M_0$-unmatched vertices forms an independent set. But then $w$ is dominated by $x_A$, $z$ and its $M_0$-matched neighbor, a contradiction. Hence, if $z$ is $M_0$-unmatched, then $x'_A$ is located by $D'\cup D^*$.

Suppose that $z$ is not an $M_0$-unmatched vertex. Thus, $z$ belongs to another bad $(D_0,M_0)$-matched $4$-cycle $B$ of $\overrightarrow{G}(S_{u,v})$ where $z=u_B$ and where $A$ is dependent on $B$
via $x'_A$ (as illustrated in Figure~\ref{fig:4cycles}). If $z\notin D'$, then both $z$ and $x'_A$ are $1$-dominated by
$D'\cup D^*$ and hence are located by $D'\cup D^*$. Finally, if $z \in D'$, then $x'_A$ is only dominated by $z$ and $x_A$ from $D' \cup D^*$. Suppose to the contrary that some other vertex $w$ is also only dominated by $z$ and $x_A$ from $D' \cup D^*$. Then, $w$ must be the neighbor of $z$ in $B$ that was removed from $D_0$, namely the vertex $w = u_B'$ (recall that $z=u_B$). Thus, $u_B'$ is adjacent to $x_A$. If $A = R$, then $x_A \in \{u,v\}$ and $x_A$ would be $3$-dominated by $D_0$, a contradiction. Hence, $A \ne R$. If $x = v$, then we contradict the fact that $v_A$ is adjacent to exactly two vertices of $D_0$, namely to $u'_A$ and $v'_A$, and therefore could not be adjacent to $u_B' \in D$. Hence, $x = u$. But then $B$ would be dependent on $A$ via $w = u_B'$. However, recall that $A$ is dependent on $B$ via $u'_A$, implying that $\overrightarrow{G}(S_{u,v})$ would contain a $2$-cycle joining $A$ and $B$, a contradiction. Hence, if $z$ is not an $M_0$-unmatched vertex, then once again $x'_A$ is located by $D'\cup D^*$. This completes the proof of Claim~\ref{clm:step1}.~$(\Box)$

\medskip
By Claim~\ref{clm:step1}, the set of $D'$-bad pairs is a proper subset of the set of $D_0$-bad pairs, implying that all remaining $D'$-bad pairs are associated with a pair of vertices of $G^*$ neither of which belong to the set $D^*$.

\noindent {\bf Step~2: We next consider all remaining \boldmath{$D'$}-bad pairs associated with a pair of vertices of \boldmath{$G^*$} neither of which belong to the set \boldmath{$D^*$}.}
For each such $D'$-bad pair $\{u,v\}$, we have $f(u,v)=\{u'',v''\}$ where
$\{u'',v''\} \subseteq V^* \setminus D^*$. Let $\cC$ be a component of $G^*-D^*$ that contains at least one edge. As observed earlier, $G^* - D^*$ has maximum degree at most~$2$. Thus, $\cC$ is a path or a cycle. If $\cC$ is a path, let $\cC$ be given by $c_0c_1 \ldots c_{k-1}$, while if $\cC$ is a cycle, let $\cC$ be given by $c_0c_1 \ldots c_{k-1}c_0$ (possibly, $\cC$ is a $2$-cycle). We now consider an edge $c_ic_{(i+1) \bmod k}$ in $\cC$, where $i \in \{0,\ldots,k-2\}$ if $\cC$ is a path and where $i \in \{0,\ldots,k-1\}$ if $\cC$ is a cycle.
Let $\{u,v\}$ be a $D_0$-bad pair such that $f(u,v)=\{c_i,c_{(i+1)\bmod k}\}$, and let $B$ be the bad $(D_0,M_0)$-matched $4$-cycle of $S_{u,v}$ such that one of $c_i$ and $c_{(i+1)\bmod k}$ is adjacent to $u'_B$ and the other to $v'_B$. Let $c_i$ be the neighbor of $x'_B$, where $x_B \in \{u_B,v_B\}$. We now propagate modifications of $D'$ along a path in $\overrightarrow{G}(S_{u,v})$ in the same way as we did in Step~1, except that we start the modifications of $D'$ along the oriented tree by replacing $x'_B$ with $x_B$ and then continuing exactly as before. The resulting modification of $D'$ ensures that for every vertex in $\cC$, at most one of its neighbors is removed from $D_0$. This process is done for all $D_0$-bad pairs associated to a pair of vertices of $G^*$ neither of which belong to the set $D^*$. Let $D''$ be the resulting modified set $D_0$.

\begin{clm}\label{clm:step2}
No $D_0$-bad pair is a $(D'' \cup D^*)$-bad pair. Further, the set of $(D''\cup D^*)$-bad pairs is a proper subset of the set of $(D'\cup D^*)$-bad pairs.
\end{clm}
\noindent\emph{Proof of Claim~\ref{clm:step2}.} It suffices to check as before the pairs of vertices that could possibly have been affected by the exchange arguments; that is, all vertices previously dominated by a vertex that has been removed from $D'$ to construct $D''$ as well as all vertices removed from $D'$ to construct $D''$. The proof is the same as in the proof of Claim~\ref{clm:step1}, except for the vertices in $G^*-D^*$. We therefore only prove that vertices that belong to components of $G^*-D^*$ that contain at least one edge are located by $D''\cup D^*$. Let $c$ be such a vertex in $G^*-D^*$. As observed earlier, such a vertex $c$ belongs to either a path component or a cycle component of $G^* - D^*$. Further, the modifications of $D'$ when constructing $D''$ ensure that for every vertex in $G^* - D^*$, at most one of its neighbors is removed from $D_0$.

We show next that $c$ was $3$-dominated by $D_0$.
Suppose to the contrary that the vertex $c$ is not $3$-dominated by $D_0$ and therefore, by definition of $\cD(M)$, is adjacent to both ends of some edge $pq$ of $M_0$. In this case, since $c$ has degree at least~$1$ in $G^* - D^*$ and therefore degree at least~$2$ in $G^*$, the edge $pq$ must be an edge $x_Bx'_B$, where $x_B \in \{u_B,v_B\}$, in a bad $(D_0,M_0)$-matched $4$-cycle $B$ of $S_{u,v}$ for some $D_0$-bad pair $\{u,v\}$. However by Claim~\ref{clm:bad-pair-structure}(c), the vertex $x'_B$ belongs to $X$. Therefore the vertex $x_B$, which is $M_0$-matched to $x_B'$, belongs to an odd component of $G - X$ that contains no $M_0$-unmatched vertex. However, the $M_0$-unmatched vertex $c$, which is adjacent to $x_B$, belongs to the same component of $G - X$ as $x_B$, a contradiction. Hence, the vertex $c$ was $3$-dominated by $D_0$.

We show now that vertex $c$ is located by $D'' \cup D^*$. If $c$ is $3$-dominated by $D'' \cup D^*$, then this follows from the twin-freeness of $G$. Hence we may assume that $c$ is not $3$-dominated by $D'' \cup D^*$. Since the vertex $c$ was $3$-dominated by $D_0$, and at most one of its neighbors is removed from $D_0$, this implies that the vertex $c$ has exactly two neighbors in $D''$ (and no neighbors in $D^*$ in $G$), and is therefore $2$-dominated by $D'' \cup D^*$. Suppose to the contrary that there is a vertex $w$ that is not located from $c$ by $D'' \cup D^*$. Let $d$ be a vertex in $D^*$ that is adjacent to $c$ in $G^*$. Then there exists a $D_0$-bad pair $\{u,v\}$ such that $f(u,v)=\{c,d\}$.

Let $B$ be the bad $(D_0,M_0)$-matched $4$-cycle of $S_{u,v}$ such that one of $c$ and $d$ is adjacent to $u'_B$ and the other to $v'_B$. Let $c$ be the neighbor of $x'_B$, where $x_B \in \{u_B,v_B\}$. By Step~1, we know that $x_B' \in D''$ and $y_B \in D''$. Therefore, the vertex $w$ must be the vertex $x_B$. Let $z$ be the neighbor of $c$ in $D''$ that is different from $x_B'$. Since the set of neighbors of $c$ in $D''$ is a subset of the set of its neighbors in $D_0$, we note that $\{x_B',z\} \subset D_0$. If $x_B = v_B$, then $v_B$ would be $3$-dominated by $D_0$, contradicting the fact that $B$ is a bad $(D_0,M_0)$-matched $4$-cycle. Similarly, if $x_B = u$, then $u$ would be $3$-dominated by $D_0$, a contradiction. Hence, $x_B = u_B$ and in $S_{u,v}$ there is a bad $(D_0,M_0)$-matched $4$-cycle that depends on $B$ via the vertex~$z$. But then $z$ has at least two neighbors apart from $c$ and $x_B$, contradicting the fact that $G$ is cubic. Therefore, the vertex $c$ is located by $D'' \cup D^*$. This completes the proof of Claim~\ref{clm:step2}.~$(\Box)$

\medskip
Claim~\ref{clm:step2} implies that there is no $(D''\cup D^*)$-bad pair. Thus, the set $D''\cup D^*$ is a locating-dominating set of $G$. Therefore,
\[
\gL(G) \le |D|+|D^*| \le \match(G)+\gamma(G^*) \le \match(G)+\frac{|V^*|}{2} \leq \match(G)+\frac{n-2\match(G)}{2} = \frac{n}{2}.
\]
This completes the proof of Theorem~\ref{t:main}.\end{proof}

\subsection{Tight examples}

We remark that the prisms $C_3 \, \Box \, K_2$ and $C_4 \, \Box \, K_2$ (shown in Figure~\ref{fig:prism}(a) and~\ref{fig:prism}(b), respectively) have location-domination number exactly one-half their order. However, it remains as an open problem to characterize all twin-free, cubic graphs $G$ of order~$n$ that satisfy $\gL(G) = \frac{n}{2}$. Note that the prisms $C_k \, \Box \, K_2$ for $k\geq 5$ do not belong to this family.

\begin{figure}[htb]
\tikzstyle{every node}=[circle, draw, fill=black!0, inner sep=0pt,minimum width=.15cm]
\begin{center}
\begin{tikzpicture}[thick,scale=.6]
  \draw(0,0) { 
    +(9.00,3.00) -- +(12.00,3.00)
    +(12.00,3.00) -- +(12.00,0.00)
    +(12.00,0.00) -- +(9.00,0.00)
    +(9.00,0.00) -- +(9.00,3.00)
    +(9.00,3.00) -- +(10.00,2.00)
    +(10.00,2.00) -- +(11.00,2.00)
    +(11.00,2.00) -- +(12.00,3.00)
    +(11.00,2.00) -- +(11.00,1.00)
    +(11.00,1.00) -- +(10.00,1.00)
    +(10.00,1.00) -- +(10.00,2.00)
    +(10.00,1.00) -- +(9.00,0.00)
    +(11.00,1.00) -- +(12.00,0.00)
    +(3.00,3.00) -- +(3.00,0.00)
    +(3.00,0.00) -- +(2.00,1.50)
    +(2.00,1.50) -- +(3.00,3.00)
    +(3.00,3.00) -- +(0.00,3.00)
    +(0.00,3.00) -- +(0.00,0.00)
    +(0.00,0.00) -- +(1.00,1.50)
    +(1.00,1.50) -- +(0.00,3.00)
    +(1.00,1.50) -- +(2.00,1.50)
    +(0.00,0.00) -- +(3.00,0.00)
    +(0.00,0.00) node{}
    +(1.00,1.50) node{}
    +(3.00,3.00) node{}
    +(3.00,0.00) node{}
    +(0.00,3.00) node{}
    +(2.00,1.50) node{}
    +(12.00,0.00) node{}
    +(11.00,1.00) node{}
    +(10.00,1.00) node{}
    +(9.00,0.00) node{}
    +(10.00,2.00) node{}
    +(11.00,2.00) node{}
    +(9.00,3.00) node{}
    +(12.00,3.00) node{}
    +(1.50,-0.75) node[rectangle, draw=white!0, fill=white!100]{(a) $C_3 \, \Box \, K_2$}
    +(10.5,-0.75) node[rectangle, draw=white!0, fill=white!100]{(b) $C_4 \, \Box \, K_2$}  
  };
\end{tikzpicture}
\end{center}
\vskip -0.6 cm \caption{The prisms $C_3 \, \Box \, K_2$ and $C_4 \, \Box \, K_2$.} \label{fig:prism}
\end{figure}
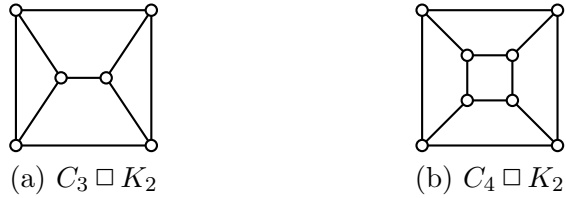

\section{Conclusion}

We conclude the paper with several intriguing open problems and questions that we have yet to solve.

\noindent
\textbf{Problem~1.}  Characterize the extremal graphs that achieve equality in the bound of Theorem~\ref{t:main}; that is, characterize the connected twin-free, cubic graphs having location-domination number exactly one-half their order.

\noindent
\textbf{Problem~2.} Determine whether the result of Theorem~\ref{t:main} can be strengthened by proving Conjecture~\ref{conj} for \emph{subcubic} graphs.

\noindent
\textbf{Problem~3.} Determine whether  Theorem~\ref{t:main} can be extended to connected cubic graphs in general (allowing twins) with the exception of a finite set of forbidden graphs. Two such forbidden graphs are the complete graph $K_4$ and the complete bipartite graph $K_{3,3}$, but it is possible that these are the only two exceptions. Proving this would still be weaker than proving the conjecture of Henning and L\"owenstein~\cite{hl12} that every cubic graph different from $K_4$ and $K_{3,3}$ has a \emph{total} locating-dominating set of size at most one-half its order.

\noindent
\textbf{Problem~4.} Determine whether every connected twin-free, cubic graph $G$ satisfies $\gL(G) \le \match(G)$. More generally, determine classes of twin-free graphs $G$ satisfying $\gL(G) \le \match(G)$. We remark that Garijo et al.~\cite{conjpaper} proved that every nontrivial twin-free graph $G$ without $4$-cycles satisfies $\gL(G) \le \match(G)$, and therefore Conjecture~\ref{conj} holds for these graphs.


\medskip
\begin{thebibliography}{99}

\bibitem{Berge} C. Berge, \textit{C. R. Acad. Sci. Paris Ser. I Math.} \textbf{247,} (1958) 258--259 and \textit{Graphs and Hypergraphs} (Chap. 8, Theorem 12), North-Holland, Amsterdam, 1973.

\bibitem{BC79indep} B.~Bollob{\'a}s and E.~J.
    Cockayne, Graph-theoretic parameters concerning domination,
    independence, and irredundance. \textit{J. Graph Theory} \textbf{3} (1979), 241--249.

\bibitem{cst} C. J. Colbourn, P. J. Slater, and L. K. Stewart. Locating-dominating sets in series-parallel networks. \textit{Congr. Numer.} \textbf{56} (1987), 135--162.

\bibitem{fh} A. Finbow and B. L. Hartnell. On locating dominating sets and well-covered graphs. \textit{Congr. Numer.} \textbf{65} (1988), 191--200.

\bibitem{fgknpv11} F.~Foucaud, E.~Guerrini, M.~Kov\v{s}e, R.~Naserasr, A.~Parreau, and P.~Valicov. Extremal graphs for the identifying code problem. \textit{European J. Combin.} \textbf{32}(4) (2011), 628--638.

\bibitem{lineID} F.~Foucaud, S.~Gravier, R.~Naserasr, A.~Parreau, and P.~Valicov. Identifying codes in line graphs. \emph{J. Graph Theory} \textbf{73}(4) (2013), 425--448.

\bibitem{Line} F.~Foucaud and M. A. Henning, Locating-dominating sets in line graphs, manuscript (2015). \url{http://arxiv.org/abs/1506.02623}

\bibitem{Heia} F.~Foucaud, M. A. Henning, C. L\"{o}wenstein, and T. Sasse, Locating-dominating sets in twin-free graphs. To appear in \textit{Discrete Applied Math}.


\bibitem{conjpaper} D. Garijo, A. Gonz\'alez and A. M\'arquez. The difference between the metric dimension and the determining number of a graph. \textit{Applied Math. Computation}  \textbf{249} (2014), 487--501.


\bibitem{hhs1} T. W. Haynes, S. T. Hedetniemi, and P. J. Slater,
    \emph{Fundamentals of Domination in Graphs}, Marcel Dekker, Inc.     New York, 1998.

\bibitem{hhs2} T. W. Haynes, S. T. Hedetniemi, and P. J. Slater
    (eds), \emph{Domination  in Graphs: Advanced Topics}, Marcel
    Dekker, Inc. New York, 1998.

\bibitem{hl12} M. A. Henning and C. L\"owenstein. Locating-total domination in claw-free cubic graphs. \textit{Discrete Math.} \textbf{312}(21) (2012), 3107--3116.


\bibitem{px82} C. Payan and N. H. Xuong. Domination-balanced graphs. \textit{J. Graph Theory} \textbf{6} (1982), 23--32.


\bibitem{o9} O. Ore, \emph{Theory of graphs}. \textit{Amer. Math. Soc. Transl.} \textbf{38} (Amer. Math. Soc., Providence, RI, 1962), 206--212.


\bibitem{rs} D. F. Rall and P. J. Slater. On location-domination numbers for certain classes of graphs. \textit{Congr. Numer.} \textbf{45} (1984), 97--106.

\bibitem{s2} P. J. Slater, Dominating and location in acyclic graphs. \textit{Networks} \textbf{17} (1987), 55--64.

\bibitem{s3} P. J. Slater, Dominating and reference sets in graphs. \textit{J. Math. Phys. Sci.} \textbf{22} (1988), 445--455.

\bibitem{s4} P. J. Slater. Locating dominating sets and locating-dominating sets. In Y.~Alavi and A.~Schwenk, editors, \emph{Graph Theory, Combinatorics, and Applications, Proc. Seventh Quad. Internat. Conf. on the Theory and Applications of Graphs} (1995), pages 1073--1079. John Wiley \& Sons, Inc.


\end{thebibliography}
\end{document}